\documentclass{amsart}
\usepackage[english]{babel}
\usepackage{amsfonts, amsmath, amsthm, amssymb,amscd,indentfirst}

\usepackage{MnSymbol}
\usepackage{color}

\usepackage{esint}
\usepackage{graphicx}
\usepackage{epstopdf}
\usepackage{amsmath,amssymb,latexsym,indentfirst}
\usepackage{times}
\usepackage{palatino}

\newtheorem{theorem}{Theorem}[section]

\newtheorem{proposition}{Proposition}[section]

\newtheorem{definition}{Definition}[section]

\newtheorem{remark}{Remark}[section]

\def\R{\mathbb{R}}

\def\d{\partial}

\def\bp{\begin{proof}}
\def\ep{\end{proof}}

\def\R{{\cal R}}

\def\R{\mathbb{R}}

\def\d{\partial}

\begin{document}

\title[]{The mass in terms of Einstein and Newton}

\author{Levi Lopes de Lima}
\address{Universidade Federal do Cear\'a (UFC),
Departamento de Matem\'{a}tica, Campus do Pici, Av. Humberto Monte, s/n, Bloco 914, 60455-760,
Fortaleza, CE, Brazil.}
\email{levi@mat.ufc.br}
\author{Frederico Gir\~ao}
\address{Universidade Federal do Cear\'a (UFC),
	Departamento de Matem\'{a}tica, Campus do Pici, Av. Humberto Monte, s/n, Bloco 914, 60455-760,
	Fortaleza, CE, Brazil.}
\email{fred@mat.ufc.br}
\author{Amilcar Montalb\'an}
\address{Universidade Federal do Cear\'a (UFC),
	Departamento de Matem\'{a}tica, Campus do Pici, Av. Humberto Monte, s/n, Bloco 914, 60455-760,
	Fortaleza, CE, Brazil.}
\email{amilcarmse@hotmail.com}
\thanks{L. L. de Lima and F. Gir\~ao have been partly supported by 
	FUNCAP/CNPq/PRONEX Grant 00068.01.00/15 and by CNPq grants 311258/2014-0 and 306196/2016-6, respectively. 
	The content of this paper was presented by the first named author at the conference ``Analytical problems in conformal geometry and applications'' (Regensburg, september/2018), funded by the SFB 1085 Higher Invariants Program/DFG. He would like to thank the organizers for the financial support. 
}

\begin{abstract}
	It is  shown that the mass of an asymptotically flat manifold with a noncompact boundary can be computed in terms of limiting surface integrals involving the Einstein tensor of the interior metric and the Newton tensor attached to the second fundamental form of the boundary. This extends to this setting previous results by several authors in the boundaryless case. The method outlined below, which is based on a coordinate-free approach due to Herzlich, also applies to asymptotically hyperbolic manifolds, again with a noncompact boundary, for which a similar notion of mass has been recently considered by Almaraz and the first named author, and both cases will be discussed here.   
\end{abstract}

\maketitle

\section{Introduction}\label{int}

Since its inception in the context of the so-called ADM formulation of General Relativity, the concept of mass, which turns out to be the fundamental numerical invariant of (time-symmetric) asymptotically flat initial data sets, has played a central role both in the original physical setting and in subsequent applications to several areas of Geometric Analysis, particularly in the study of questions related to the existence and multiplicity of solutions of the classical Yamabe problem; see \cite{BM, dLPZ} and the references therein. However, since the mass is costumarily  defined by a limiting process involving integration over larger and larger spheres of a certain quantity depending on the derivatives of the metric up to first order with no manifest physical or geometric meaning, it is not at all obvious from its very definition how it ties to the asymptotic geometry of the given Riemannian metric. 
This lack of an explicit geometric meaning for the mass naturally suggests the problem of expressing  it as an asymptotic surface integral of geometric quantities directly related to the underlying Riemannian structure.  This may be achieved by first establishing the expected relationship between the mass (as well as other asymptotic invariants like the center of mass, etc.) and certain curvature integrals of the given metric for a special class of initial data sets and then appealing to a suitable density theorem making sure that such a class is dense in the space of all apropriate inital data sets satisfying the relevant dominant energy condition in a topology in which the given invariant is continuous. 
Besides confirming the general validity of a geometric formula for the mass already available in  the physics literature \cite{AH, Ch}, this approach has been widely used in many settings and has provided considerable insight on the nature of a large class of such asymptotic invariants \cite{CW,CS,Hu}.

We remark, however, that this method  of ``improving the asymptotics'' is rather technical in nature, as it necessarily involves the previous establishment of a corresponding density theorem in suitable weighted functional spaces. Thus, it is highly desirable to circumvent the use of this piece of hard analysis by providing a direct proof of the geometric formulae for the asymptotic invariants. In the specific case of the ADM mass and the center of mass for asymptotically flat manifolds, this has been accomplished in \cite{MT} through an ingenious (but rather lengthy) computation in local coordinates. Even more recently, by combining the penetrating analysis  on the existence and well-definiteness of these asymptotic invariants due to Michel \cite{Mi} with a Pohozaev-Schoen-type integral formula involving the Ricci and scalar curvatures, Herzlich \cite{H} was able to present an elementary approach to the problem which in particular dispenses local computations and hence provides a manifestly conceptual definition of the invariants in geometric terms. We mention that, as already explained in \cite{H}, this approach also applies to  asymptotically hyperbolic manifolds, thus furnishing a geometric formula for the mass functional attached to such manifolds in \cite{CH}.   

The purpose of this note is to show that the elegant approach in \cite{H} can be adapted to express the mass and the center of mass for the class of  asymptoticaly flat manifolds carrying a noncompact boundary recently studied in \cite{ABdL} in terms of the Einstein tensor of the metric in the interior, the Newton tensor attached to the second fundamental form of the boundary and  suitable asymptotically conformal vector fields (Theorem \ref{flatmain}). The method of proof, which as in \cite{H} combines Michel's analysis with a {\em generalized} Pohozaev-Schoen integral formula (Proposition \ref{poho}) which we believe might have an independent interest, also allows us to express the mass functional associated to an asymptotically hyperbolic manifold with a noncompact boundary as defined in \cite{AdL}, again in terms of the Einstein and Newton tensors (Theorem \ref{masshyp}).

\section{The asymptotically flat case}\label{assflat}

We start by recalling the main definition in \cite{ABdL}. In the following, if $\mathbb R^n_+=\{x\in\mathbb R^n;x_n\geq 0\}$ is the Euclidean half-space endowed with the standard flat metric $\delta$ and $r=\sqrt{x_1^2+\cdots+x_n^2}$ is the standard radial coordinate, then we set $\mathbb R^n_{r_*,+}=\{x\in\mathbb R^n_+;r\geq r_*\}$, where $r_*>0$. Moreover, $(M,g)$ will denote an oriented, smooth manifold of dimension $n\geq 3$ carrying a non-compact boundary $\Sigma$. Also, $R^g$ is the scalar curvature of $g$ and $H^g$ is the mean curvature of $\Sigma$, computed with respect to the {\em inward} unit normal vector field $\nu^g$. In general, we will denote by the same symbol the restriction to $\Sigma$ of a metric on $M$. The next definition isolates the class of manifolds we are interested in; they are modelled at infinity on the Euclidean half-space $\mathbb R^n_{+}$ above.  

\begin{definition}\label{assflatdef}\cite{ABdL}
	We say that $(M,g)$ as above is {\em asymptotically flat} (with a noncompact boundary $\Sigma$) if there exists a compact subset $A\subset M$ and a diffeomorphism $F:M\backslash A\to\mathbb R^n_{r_0,+}$, for some $r_0>0$, so that:
	\begin{enumerate}
		\item as $r\to +\infty$,
		\begin{equation}\label{extracond}
		|g_{ij}-\delta_{ij}|+r|\partial_kg_{ij}|+r^2|\partial_l\partial_kg_{ij}|=O(r^{-\tau}), \quad \tau>\frac{n-2}{2},
		\end{equation}
		where the Euclidean (and hence the radial) coordinates have been transplanted to $M\backslash A$ by means of the chart at infinity $F$ and used to compute the coeffients of $g$ and its derivatives;
		\item both $\int_MR^gd{\rm vol}_M^g$ and $\int_\Sigma H^gd{\rm vol}_\Sigma^g$ are finite. 
		\end{enumerate} 
	\end{definition}

As already observed, in the presence of a chart at infinity $F$ we may identify $M\backslash A$ with $\mathbb R^n_{r_0,+}$. This allows us to define
for $r_0<r<r'$,
$
M_{r,r'}=\{x\in M\backslash A;r\leq |x| \leq r'\}$,
$\Sigma_{r,r'}=\{x\in \partial (M\backslash A);r\leq |x| \leq r'\}
$ and the coordinate $(n-1)$-hemisphere
$S^{n-1}_{r,+}=\{x\in M\backslash A|;|x|=r\}$, 
so that 
$$\d M_{r,r'}=S^{n-1}_{r,+}\cup \Sigma_{r,r'}\cup S^{n-1}_{r',+}.$$
We represent by $\mu^\delta$ the outward unit normal vector field to $S^{n-1}_{r,+}$, computed with respect to the reference metric $\delta$. Also, we consider the coordinate $(n-2)$-sphere $S^{n-2}_r=\partial S_{r,+}^{n-1}\subset \partial(M\backslash K)$, endowed with its outward unit conormal vector field $\vartheta^\delta$, which is tangent to $\Sigma$.
As before we  set $e=g-\delta$, where we have written $\delta=F^*\delta$ for simplicity of notation, and we define the $1$-form
\begin{equation}\label{charge}
\mathbb U^\delta_{e,w}=w({\rm div}^\delta e-d{\rm tr}^\delta e) -\nabla^\delta w\righthalfcup e+{\rm tr}^\delta e\,dw,
\end{equation}
where $w$ is a function on $\mathbb R^n_+$.
In the following we will make use of the universal constants
\[
c_n=\frac{1}{2(n-1)\omega_{n-1}}, \quad d_n=\frac{1}{(2-n)(n-1)\omega_{n-1}},
\]
where $\omega_{n-1}$ is the volume of the unit $(n-1)$-sphere.
Finally, we represent by $\eta^\delta=-\nu^\delta$ the {\rm outward} unit normal vector field along $\Sigma$.
The following result, proved in \cite{ABdL}, provides the most relevant asymptotic invariant for the class of manifolds appearing in Definition \ref{assflatdef}.  

\begin{theorem}\label{flatmass}
	If $(M,g)$ is asymptotically flat  as above then the quantity 
	\begin{equation}\label{flatmass2}
	\mathfrak m_{(M,g)}=c_n\lim_{r\to +\infty}\left[\int_{S^{n-1}_{r,+}} \mathbb U^\delta_{e,{ 1}}(\mu^\delta) d{\rm vol}_{S^{n-1}_{r,+}}^\delta-\int_{S^{n-2}_r}e(\eta^\delta,\vartheta^\delta)d{\rm vol}_{S^{n-2}_r}^\delta\right]
	\end{equation}
	is finite and its value does not depend on which chart at infinity is chosen, where ${1}$ here denotes the function identically equal to $1$.
	\end{theorem}

\begin{remark}\label{secorder}
	{\rm 	In order to have the mass well defined it is not required to assume a second order pointwise  control on the metric as in (\ref{extracond}) above; a first order control suffices (see \cite{Mi} for a discussion of this rather subtle point in the boundaryless case). However, for our purposes, this extra assumption is crucial as it not only implies that $H^g=O(r^{-\tau-1})$ but also that both ${\rm Ric}^g$ and $R^g$ are $O(r^{-\tau-2})$. 
	}
\end{remark}

This expression for the mass $\mathfrak m_{(M,g)}$ involves integration over larger and larger (hemi-)spheres  of quantities depending on the derivatives of the metric up to first order with no direct geometric meaning. Thus, as already remarked in the Introduction, it is highly desirable to obtain an expression for this invariant in terms of more fundamentally geometric objects. In the boundaryless case, it is well-known that the mass can be asymptotically  written in terms of the Einstein tensor of $g$, 
\[
E^g={\rm Ric}^g-\frac{R^g}{2}g;
\] 
see \cite{H} and the references therein.
Here we show how the elegant method in \cite{H} can be adapted to express $\mathfrak m_{(M,g)}$ above in terms of 
the Einstein
tensor
$E^g$ and the Newton tensor along the boundary, namely,
\[
J^g=\Pi^g-H^gg, 
\] 
where $\Pi^g$ is the second fundamental form of $\Sigma$, defined with respect to $\nu^g$, the outward unit normal, and  $H^g={\rm tr}_{g|_{\Sigma^g}}\Pi^g$ is the mean curvature. As in the boundaryless case, these tensors should be evaluated on 
the radial vector field $X_0=r\partial_r$.
More precisely, the following result holds. 

\begin{theorem}\label{flatmain} One has 
	\[
	\mathfrak m_{(M,g)}=d_n\lim_{r\to +\infty}\left[\int_{S^{n-1}_{r,+}}E^g(X_0,\mu^g)d{\rm vol}_{S^{n-1}_{r,+}}^g+\int_{S^{n-2}_r}J^g(X_0,\vartheta^g)d{\rm vol}_{S^{n-2}_r}^g \right].
	\]
	\end{theorem}

A simple variation of the procedure leading to the proof of Theorem  \ref{flatmass} also allows us to define the {\em center of mass} of an asymptotically flat manifold $(M,g)$ as in Definition \ref{assflatdef}; see \cite{Mi} for a  discussion of the boundaryless case. 

\begin{theorem}\label{centermass}
	Let $(M,g)$ be an asymptotically flat manifold with $\mathfrak m_{(M,g)}\neq 0$ and assume moreover that  both $\int_MrR^gd{\rm vol}_M^g$ and $\int_\Sigma rH^gd{\rm vol}_\Sigma^g$ are finite, where the asymptotic radial coordinate $r$ has been smoothly extended to the whole of $M$.
	Then for each $\alpha=1,\cdots,n-1$ the quantity
	\[
	\mathfrak c^{\alpha}_{(M,g)}=\frac{c_n}{\mathfrak m_{(M,g)}}\lim_{r\to +\infty}\left[\int_{S^{n-1}_{r,+}}\mathbb U^\delta_{e,x_{\alpha}}(\mu^\delta)d{\rm vol}_{S^{n-1}_{r,+}}^\delta-\int_{S^{n-2}_r}x_\alpha e(\eta^\delta,\vartheta^\delta)d{\rm vol}_{S^{n-2}_r}^\delta\right]
	\]	
	is finite. Moreover, the vector $\mathfrak c_{(M,g)}=(\mathfrak c^1_{(M,g)},\cdots,\mathfrak c^{n-1}_{(M,g)})$ does not depend on the chosen chart at infinity (up to Euclidean rigid motions) and is termed the {\em center of mass} of $(M,g)$.  
\end{theorem}

\begin{remark}\label{regge}
	{\rm It is known that the center of mass $\mathfrak c_{(M,g)}$ is also well defined if instead of the integrability conditions on $R^g$ and $H^g$ in Theorem \ref{centermass} above we assume the corresponding Regge-Teitelboim (RT) evenness conditions at infinity. More precisely,  if $x=(x',x_n)\in M\backslash A$, where $x'=(x_1,\cdots,x_{n-1})$, define the involution $\tau:M\backslash A\to M\backslash A$ by $\tau(x',x_n)=(-x',x_n)$, so that for each function $f$ on $M\backslash A$ we may consider its {\em odd part}
		\[
		f^{{\rm odd}}(x)=\frac{1}{2}\left(f(x)-f(\tau x)\right). 
		\] 
		Then the RT requirement is that
		\[
		|g_{ij}^{\rm odd}|+r |\partial_kg_{ij}^{\rm odd}|=O(r^{-\tau-1}),  \quad (R^g)^{\rm odd}=O(r^{-2\tau-2}), \quad (H^g)^{\rm odd}=O(r^{-2\tau-1});
		\] 
	see \cite{CCS, CW,Hu,H} for discussions in the boundaryless case.}
	\end{remark}
 
In order to write the center of mass in Theorem \ref{centermass} in terms of $E^g$ and $J^g$ we make use of the conformal vector fields $X_{\alpha}=r^2\partial_\alpha-2x_\alpha\sum_ix_i\partial_i$, $\alpha=1,\cdots,n-1$, which obviously are tangent to $\partial \mathbb R^n_+$ (similarly to $X_0$). 
\begin{theorem}\label{centerflat}
	With the notation above, 
	\[
	\mathfrak c^\alpha_{(M,g)}=-\frac{d_n}{2\mathfrak m_{(M,g)}}\lim_{r\to +\infty}\left[\int_{S^{n-1}_{r,+}}E^g(X_\alpha,\mu^g)d{\rm vol}_{S^{n-1}_{r,+}}^g+\int_{S^{n-2}_{r}}J^g(X_\alpha,\vartheta^g)d{\rm vol}_{S^{n-2}_{r}}^g\right].
	\]	
\end{theorem}

\section{The proofs of Theorems \ref{flatmain} and \ref{centerflat}}\label{proofflat}

We  follow \cite{H} closely and for $r>4r_0$  consider a cut-off function $\chi=\chi_r(r')$ on $M\backslash A$ which vanishes for $r'\leq r/2$, equals $1$ for $r'\geq 3r/4$ and satisfies the estimates 
\[
|\nabla \chi|\leq Cr^{-1},\quad |\nabla^2\chi|\leq Cr^{-2}, \quad |\nabla^3\chi|\leq Cr^{-3},
\]
for some $C>0$ independent of $r$. We then define a metric on $M_r:=M_{r/4,r}$ by
\begin{equation}\label{herzmet}
h=(1-\chi)\delta+\chi g,
\end{equation}
which is then extended to the whole of $M$ in the obvious manner. With this notation at hand, the computations in \cite[Section 3]{ABdL} leading to the proof of Theorem \ref{flatmass} above easily imply the following alternative expression for the mass. We remark that this also follows by adapting to this setting Michel's analysis \cite{Mi} in the boundaryless case.

\begin{proposition}\label{altern}
	With the notation above, 
	\begin{equation}\label{altern2}
	\mathfrak m_{(M,g)}=c_n\lim_{r\to +\infty}\left[\int_{M_r}R^h d{\rm vol}_{M_r}^\delta+2\int_{\Sigma_{r}}H^hd{\rm vol}^\delta_{\Sigma_{r}}\right],
	\end{equation}
	where $\Sigma_r:=\Sigma_{r/4,r}$.
	\end{proposition}

The next ingredient in the proof of Theorem \ref{flatmain} is a Pohozaev-Schoen-type integral identity whose infinitesimal version we present now. A special case of this result appears in \cite[Lemma 2.1]{H}; see also \cite{BdLF} for another variant of this useful identity. We have chosen here to  present this  material in full generality as we believe it might be useful in other contexts as well. Thus, let us     
consider a Riemannian $p$-manifold $(N,\gamma)$ and  take $K=K_{ij}\in \mathcal S^2(N,\gamma)$ to be a symmetric, twice covariant tensor and $Y=Y^i\in\mathcal X(N)$ a vector field, where $i,j=1,\cdots,p$.

\begin{proposition}\label{poho}
	There holds
	\begin{equation}\label{poho2}
	{\rm div}^\gamma (Y\righthalfcup K)=\langle{\rm div}^\gamma K,Y\rangle_\gamma+\langle K,\widetilde{{\rm div}^\gamma_\dagger Y}\rangle_\gamma+\frac{1}{p}{\rm div}^\gamma Y{\rm tr}^\gamma K,
	\end{equation} 
	where ${\rm div}^\gamma_{\dagger}: \mathcal X(N)\to \mathcal S^2(N,\gamma)$ is the $L^2$ adjoint of the divergence map ${\rm div}^\gamma:\mathcal S^2(N,\gamma)\to  \mathcal X(N)$ and the tilde means the trace free part.
	\end{proposition}

\begin{proof}
Computing at the center of a normal coordinate system we have $(Y\righthalfcup K)_j=K_{ij}Y_i$ and hence
\begin{eqnarray*}
{\rm div}^\gamma (Y\righthalfcup K) & = & (K_{ij}Y_i)_{,j}\\
& = & K_{ij,j}Y_i+K_{ij}Y_{i,j}\\
& = & K_{ij,j}Y_i+\frac{1}{2}K_{ij}(Y_{i,j}+Y_{j,i}),
\end{eqnarray*}
where the comma means covariant derivation.
In invariant terms this means that
\[
{\rm div}^\gamma (Y\righthalfcup K)=\langle{\rm div}^\gamma K,Y\rangle_\gamma+\frac{1}{2}\langle K,\mathcal L_Y\gamma\rangle_\gamma, 
\]
where $\mathcal L$ is the Lie derivative.
Since
\[
{\rm div}^\gamma_\dagger Y=\frac{1}{2}\mathcal L_Y\gamma,
\]
we get 
\[
{\rm div}^\gamma (Y\righthalfcup K)=\langle{\rm div}^\gamma K,Y\rangle_\gamma+\langle K,{\rm div}^\gamma_\dagger Y\rangle_\gamma. 
\]
Also, since 
\[
({\rm div}^\gamma_\dagger Y)_{ij}=\frac{1}{2}(Y_{i,j}+Y_{j,i}), 
\]
we have 
\[
{\rm tr}^\gamma{\rm div}^\gamma_\dagger Y={\rm div}^\gamma Y, 
\]
so that 
\[
\widetilde{{\rm div}^\gamma_\dagger Y}={{\rm div}^\gamma_\dagger Y}-\frac{1}{p}({\rm div}^\gamma Y)\gamma ,
\]
and the result follows.
\end{proof}


To proceed with the proof of Theorem \ref{flatmain},
we integrate (\ref{poho2}) over the $n$-dimens\-ional half-annulus $(M_r,h)$, where $h$ is the interpolating metric in (\ref{herzmet}). 
We take $K$ to be 
\[
E^h={\rm Ric}^h-\frac{R^h}{2}h,
\]  
the Einstein tensor of $h$, so that 
\[
{\rm div}^hE^h=0,\quad {\rm tr}^hE^h=\frac{2-n}{2}R^h.
\]
Noticing that 
\[
\int_{S^{n-1}_{r/4,+}}E^h(Y,\eta^h)d{\rm vol}_{S^{n-1}_{r/4,+}}^h=0
\]
because $h=\delta$ there, we obtain
\begin{eqnarray*}
\int_{S^{n-1}_{r,+}}E^g(Y,\mu^g)d{\rm vol}_{S^{n-1}_{r,+}}^g+\int_{\Sigma_{r}}E^h(Y,\eta^h)d{\rm vol}_{\Sigma_r}^h
& = & 
\int_{M_r}\langle E^h,\widetilde{{\rm div}^h_\dagger Y}\rangle_hd{\rm vol}_{M_r}^h\\
& & \quad +\frac{2-n}{2n}\int_{M_r}{\rm div}^hY\cdot R^hd{\rm vol}_{M_r}^h.
\end{eqnarray*}

We now take $Y$ to be the radial vector field $X_0=r\partial_r$ so that 
\begin{equation}\label{cruc1}
{\rm div}^\delta X_0=n,\quad \widetilde{{\rm div}^\delta_\dagger X_0}=0,
\end{equation}
the last identity holding due to the fact that $X_0$ is conformal relatively to $\delta$. As in \cite{H} we note that, as $r\to +\infty$,
\begin{eqnarray*}
\int_{M_r}{\rm div}^hX_0\cdot R^hd{\rm vol}_M^h 
& = & \int_{M_r}{\rm div}^\delta X_0\cdot R^hd{\rm vol}_M^\delta+o(1)\\
& = & n\int_{M_r}R^hd{\rm vol}_M^\delta+o(1),
\end{eqnarray*}
where we used the decay on the scalar curvature described in Remark \ref{secorder}.
Similarly, 
\begin{eqnarray*}
\int_{M_r}\langle E^h,\widetilde{{\rm div}^h_\dagger X_0}\rangle_hd{\rm vol}_M^h
& = & \int_{M_r}\langle E^h,\widetilde{{\rm div}^\delta_\dagger X_0}\rangle_hd{\rm vol}_M^\delta+o(1)\\
& = & o(1),
\end{eqnarray*}
which finally gives
\begin{equation}\label{halfannul}
\int_{S^{n-1}_{r,+}}E^g(X_0,\mu^g)d{\rm vol}_{S^{n-1}_{r,+}}^g+\int_{\Sigma_{r}}E^h(X_0,\eta^h)d{\rm vol}_\Sigma^h=
\frac{2-n}{2}\int_{M_r}R^hd{\rm vol}_{M}^\delta+o(1).
\end{equation}
	
We now integrate (\ref{poho2}) over another configuration, namely, the $(n-1)$-dimens\-ional annulus  $\Sigma_{r}$, whose boundary is $\partial \Sigma_r=S^{n-2}_{r/4}\cup S^{n-2}_{r}$. This time we take $K$ to be $J^h=\Pi^h-H^{h}h$, the Newton tensor of $\Sigma$ with respect to $h$, so that 
\[
{\rm tr}^h J^h=(2-n)H^{h},\quad {\rm div}^h J^h={\rm Ric}^h(\eta,\cdot),
\] 
with the last identity being just the contracted Codazzi equation; see the proof of \cite[Theorem 14]{BdLF}. We also take $Y$ to be $r\partial_r|_{\Sigma_r}$, which we still denote by $X_0$. Hence, 
\begin{equation}\label{cruc2}
{\rm div}^\delta X_0=n-1, \quad \widetilde{{\rm div}^\delta_\dagger X_0}=0,
\end{equation} 
the last identity being true due to the fact that $X_0|_\Sigma$ is conformal relatively to $\delta=\delta|_{\Sigma}$. Using that
\[
\int_{S^{n-2}_{r/4}}J^h(X,\vartheta^h)d{{\rm vol}}_{S^{n-2}_{r/4}}^h=0, 
\]
because $h=\delta$ there, we get
\begin{eqnarray*}
\int_{S_{r}^{n-2}}J^g(X_0,\vartheta^g)d{{\rm vol}}_{S^{n-2}_{r}}^g
& = &  \int_{\Sigma_{r}}{\rm Ric}^h(X_0,\eta^h)d{\rm vol}_{\Sigma_r}^h\\
& & \quad +\int_{\Sigma_{r}}\langle \widetilde{{\rm div}^h_\dagger X_0},J^h\rangle_h d{\rm vol}_{\Sigma_r}^h\\
& & \quad\quad +\frac{2-n}{n-1}\int_{\Sigma_{r}}{\rm div}^h X_0\cdot H^h d{\rm vol}_{\Sigma_r}^h.
\end{eqnarray*}
As before, we make use of the decay on the mean curvature in Remark \ref{secorder} to get
\begin{eqnarray*}
\int_{\Sigma_{r}}{\rm div}^h X_0\cdot H^{h}d{\rm vol}_{\Sigma_r}^h
& = & \int_{\Sigma_{r}}{\rm div}^\delta X_0\cdot H^{h}d{\rm vol}_{\Sigma_r}^\delta+o(1)\\
& = & (n-1) \int_{\Sigma_{r}}H^{h}d{\rm vol}_{\Sigma_r}^\delta + o(1), 	
	\end{eqnarray*}
and similarly,
\begin{eqnarray*}
\int_{\Sigma_{r}}\langle \widetilde{{\rm div}^h_\dagger X_0},J\rangle_h d{\rm vol}_{\Sigma_r}^h
& = & \int_{\Sigma_{r}}\langle \widetilde{{\rm div}^\delta_\dagger X_0},J\rangle_h d{\rm vol}_{\Sigma_r}^\delta+o(1)\\
& = & o(1), 	
	\end{eqnarray*}
so that 
\begin{equation}\label{annul}
\int_{S^{n-2}_{r}}J^g(X_0,\vartheta^g)d{\rm vol}_{S^{n-2}_{r}}^g
 =   \int_{\Sigma_{r}}{\rm Ric}^h(X_0,\eta^h)d{\rm vol}_{\Sigma_r}^h+\frac{2-n}{2}\int_{\Sigma_{r}}2H^{h}d{\rm vol}_{\Sigma_r}^h+ o(1).
\end{equation}
Putting together (\ref{halfannul}) and (\ref{annul}) 
and using that 
\[
E^h(X_0,\eta^h)-{\rm Ric}^h(X_0,\eta^h)=-\frac{R^h}{2}\langle X_0,\eta^h\rangle_h=0
\]
because $X_0$ is tangent to $\Sigma$, we end up with 
\begin{eqnarray*}
d_n\left[	\int_{S^{n-1}_{r,+}}E^g(X_0,\mu^g)d{\rm vol}_{S^{n-1}_{r,+}}^g+\int_{S^{n-2}_{r}}J^g(X_0,\vartheta^g)d{\rm vol}_{S^{n-2}_{r}}^g\right]&&\\
=c_n\left[\int_{M_r}R^hd{\rm vol}_{M_r}^\delta+2\int_{\Sigma_r}H^{h}d{\rm vol}_{\Sigma_r}^\delta\right]+o(1),
\end{eqnarray*}
so Theorem \ref{flatmain} follows from Proposition \ref{altern}.

The proof of Theorem \ref{centerflat} is obtained by essentially the same argument as above, where we now use  that
\begin{equation}\label{cruc3} 
{\rm div}^\delta X_\alpha=-2nx_\alpha,\quad \widetilde{{\rm div}^\delta_\dagger X_\alpha}=0,
\end{equation} 
the last identity holding due to the fact that $X_\alpha$ is conformal relatively to $\delta$.

\begin{remark}\label{alth}
	{\rm Although not explicitly mentioned so far, a crucial ingredient in the proofs of the  theorems above is the fact that the reference space $\mathbb R^n_+$ is {\em static} in the sense that 
	\begin{equation}\label{statdelta}
	\mathcal N_{\delta}:=\left\{w:\mathbb R^n_+\to \mathbb R; (\nabla^\delta)^2w=0, \frac{\partial w}{\partial\eta^\delta}=0\right\}
	\end{equation}
is non-trivial, where $\eta^\delta=(0,0,\cdots,-1)$ is the outward unit normal along $\partial\mathbb R^n_+$. In fact, one easily checks that $\mathcal N_\delta$ is generated by the functions
\[
w_0(x)=1, w_1(x)=x_1,\cdots, w_{n-1}(x)=x_{n-1}. 
\]
Using this terminology, the key point is that  ${\rm div}^\delta X\in \mathcal N_\delta$ whenever $X$ is conformal with respect to $\delta$ and everywhere tangent to $\partial \mathbb R^n_+$; see \cite[Lemma 2.2]{H} for the boundaryless version of this result and compare with the first identity in (\ref{cruc1}) and (\ref{cruc3}) above. In fact, the Hessian operator $(\nabla^\delta)^2$ above equals $D{R}^\delta_\dagger$, the $L^2$ adjoint of the linearization $D{R}^\delta$ of the scalar curvature operator at $\delta$. In general, we say that a Riemannian manifold $(N,\gamma)$, possibly endowed with a noncompact totally geodesic boundary, is {\em static} if 
\begin{equation}\label{statk}
\mathcal N_\gamma:=\left\{w:N\to\mathbb R;D{R}^\gamma_\dagger w=0,\frac{\partial w}{\partial \eta^\gamma}=0\right\}
\end{equation}
is non-trivial, where as always $\eta^\gamma$ is the outward unit normal along the boundary; see \cite{AdL} for an application of this concept in the asymptotically hyperbolic case. Since 
\[
D{R}^\gamma_\dagger w=(\nabla^\gamma)^2w +(\Delta^\gamma w)\gamma-w{\rm Ric}^\gamma,
\]
we see that (\ref{statk}) reduces to (\ref{statdelta}) when $\gamma=\delta$.}
	\end{remark}

\section{The asymptotically hyperbolic case}\label{asymcase}

As already observed, the methods above may be used to express in geometric terms the mass of an asymptotically hyperbolic manifold with a noncompact boundary, an asymptotic invariant studied in detail in \cite{AdL}. Here the reference space is 
the {\em hyperbolic half-space} in dimension $n\geq 3$  given by 
$$
\mathbb H^n_+=\{y\in\R^{1,n}; y_0>0, \langle y,y\rangle_L=-1, y_n\geq 0\},
$$ 
where $\R^{1,n}$ is the Minkowski space with the standard flat metric 
\[
\langle y,y\rangle_L=-y_0^2+y_1^2+\cdots+y_n^2.
\] 
By setting $s=\sqrt{y_1^2+\cdots+y_n^2}$, $\mathbb H^n_+$ inherits the hyperbolic
metric 
$$
b=\frac{ds^2}{1+s^2}+s^2h_0,
$$
where $h_0$ stands for the canonical metric on the unit upper hemisphere $\mathbb S^{n-1}_+$.
Note that $\mathbb H^n_+$
carries a noncompact, totally geodesic boundary, namely, $\partial \mathbb H^n_+=\{y\in\mathbb H^n_+;y_n=0\}$.
Similarly to Remark \ref{alth} above, we set
\[
\mathcal N_b:=\left\{W:\mathbb H^n_+\to\mathbb R;(\nabla^b)^2W=nW,\frac{\partial W}{\partial \eta^b}=0  \right\},
\]
where, as usual, $\eta^b$ is the outward unit normal to $\partial \mathbb H^n_+$. This space is spanned by the functions  $W_a(y)=y_a|_{\mathbb H_+^n}$, $a=0,1,\cdots,n-1$, so $\mathbb H^n_+$ is a {\em static} space; see again Remark \ref{alth}.

We now define the  notion of an asymptotically hyperbolic manifold with a non-compact boundary having $(\mathbb H^n_+,b)$ as a model at infinity; see \cite{AdL} for further details.
For this it is convenient to set $s=\sinh \rho$, so that 
\[
b=d\rho^2+\sinh^2\rho\, h_0
\]
in the coordinates $(\rho,\theta)$, $\theta\in \mathbb S^{n-1}_+$. Notice that in this model $\rho$ is the geodesic distance from the origin $\rho=0$.

\begin{definition}\label{def:as:hyp}\cite{AdL}
	We say that $(M^n,g)$ is {\it{asymptotically hyperbolic}} (with a non-compact boundary $\Sigma$) if there exists $\rho_0>0$, a compact subset $A\subset M$ and a diffeomorphism
	$
	F:M\backslash A\to\mathbb H_{+,\rho_0}^n
	$
	such that: 
	\begin{enumerate}
		\item As $\rho\to+\infty$,  
		\begin{equation}\label{asympthyp}
	|g_{ij}-\delta_{ij}|+|\mathfrak f_kg_{ij}|+ |\mathfrak f_l\mathfrak f_kg_{ij}|=O(e^{-\rho\tau}), \quad \tau>\frac{n}{2},
		\end{equation}
		where $\{\mathfrak f_i\}_{i=1}^n$ is a $b$-orthonormal frame with $\mathfrak f_1=\partial_\rho$ and $g_{ij}=\langle F^*\mathfrak f_i,F^*\mathfrak f_j\rangle_g$;
		\item both $\int_Me^\rho(R^g+n(n-1))d{\rm vol}_M^g$ and $\int_\Sigma e^\rho H^gd{\rm vol}_\Sigma^g$ are finite, where the asymptotical radial coordinate $\rho$ has been smoothly extended to the whole of $M$.
	\end{enumerate}
\end{definition} 

\begin{remark}\label{secorder2}{\rm The same observations as in Remark \ref{secorder} above apply here, so in particular we have that ${\rm Ric}^g+(n-1)g$, $R^g+n(n-1)$ and $H^g$ are $O(e^{-\rho\tau})$.}
	\end{remark}

Regarding this class of manifolds, the next result has been proved in \cite{AdL}. It provides any such manifold with an asymptotic invariant which naturally extends the corresponding notion in the boundaryless case treated in \cite{CH, Wa}. 

\begin{theorem}\label{masshyp}
	If $(M,g)$ is an asymptotically hyperbolic manifold as above then the linear map $\mathfrak M_{(M,g)}:\mathcal N_b\to\mathbb R$ given by
	\begin{equation}\label{masshyp2}
 \mathfrak M_{(M,g)}(W):=c_n\lim_{\rho\to +\infty}\left[\int_{S^{n-1}_{\rho,+}}\mathbb U^b_{e,W}(\mu^b)d{\rm vol}_{S^{n-1}_{\rho,+}}^b-\int_{S^{n-2}_\rho}W e(\eta^b,\vartheta^b)d{\rm vol}_{S^{n-2}_\rho}^b\right],
	\end{equation}
	where here $e=g-b$,
	is well defined and independent of the chosen chart at infinity. Here, 
	$S_{\rho,+}^{n-1}$, etc, are defined just as in the flat case. 
	\end{theorem}

We mention that a positive mass teorem for the mass functional $\mathfrak M$ above has been established in \cite{AdL}
under the assumption that the underlying manifold is spin. 

With this notion at hand, it is not hard to use a simple variation of the arguments in the pevious section to check that the mass functional can be asymptotically expressed in terms of the Einstein and Newton tensors. More precisely, define the modified Einstein tensor
\[
\widehat E^g=E^g-\frac{(n-1)(n-2)}{2}g, 
\] 
and observe that the vector fields $X_a$, $a=0,1,\cdots,n-1$ considered above, when properly transplanted to our hyperbolic model $(\mathbb H^n_+,b)$, are conformal with respect to $b$. Thus, $\widetilde{{\rm div}^b_\dagger X_a}=0$ and a direct computation shows that ${\rm div}^bX_a=nW_a$, in accordance with Remark \ref{alth}. Putting these facts together and proceeding exactly as above we easily obtain the following result.

\begin{theorem}\label{asymhypl}Under these conditions,  for any  $a=0,1,\cdots,n-1$ there holds
	\[
	\mathfrak M_{(M,g)}(W_a)=d_n\lim_{\rho\to +\infty}\left[\int_{S^{n-1}_{\rho,+}}\widehat E^g(X_a,\mu^g)d{\rm vol}_{S^{n-1}_{\rho,+}}^g+\int_{S^{n-2}_\rho}J^g(X_a,\vartheta^g)d{\rm vol}_{S^{n-2}_\rho}^g \right].
	\]
	\end{theorem}

It should be mentioned that applications of this result to rigidity questions related to asymptotically hyperbolic Einstein manifolds with a noncompact boundary appear in \cite{AdL}.

\end{document}